\documentclass[11pt]{amsart}
\usepackage{amssymb}
\usepackage{graphicx,tikz}
\usetikzlibrary{arrows}
\date{}

\graphicspath{{.}{./eps/}}

\newtheorem{theorem}{Theorem}
\newtheorem{corollary}[theorem]{Corollary}

\newtheorem{definition}{Definition}

\newtheorem{proposition}[theorem]{Proposition}
\newtheorem{lemma}[theorem]{Lemma}

\newcommand\GG{\Gamma_d(q)}

\newcommand\dll{DL_d(q)}

\newcommand\Z{\mathbb Z}

\newcommand\N{\mathbb N}
\newcommand\mlg{\left((m_1(g),l_1(g)),(m_2(g),l_2(g)), \cdots ,(m_d(g),l_d(g))\right)}
\newcommand\mlq{\left((m_1, l_1), (m_2, l_2), \dots, (m_d, l_d)\right)}
\newcommand\eone{\textbf e_1}
\newcommand\etwo{\textbf e_2}

\newcommand\ed{\textbf e_d}
\newcommand\edd{\textbf e_{d-1}}

\newcommand\si{A_{\sigma(i)}}

\newcommand\ti{A_{\tau(i)}}

\title{Metric properties of Diestel-Leader groups}
\author[Melanie Stein] {Melanie Stein}
\address{Department of Mathematics, Trinity College, Hartford, CT 06106}
\email{melanie.stein@trincoll.edu}

\author[Jennifer Taback] {Jennifer Taback}
\address{Department of Mathematics, Bowdoin College, Brunswick, ME 04011}
\email{jtaback@bowdoin.edu}

\thanks{The second author acknowledges support from
National Science Foundation grants DMS-0604645 and DMS-1105407. Both authors acknowledge support from a Bowdoin College
Faculty Research Award. Many thanks to Sean Cleary, David Fisher, Martin Kassabov, Tim Riley, Peter Wong and Kevin Wortman
for helpful conversations during the writing of this paper.} 

\date{\today}
\begin{document}

\begin{abstract}
In this paper we investigate metric properties of the groups $\Gamma_d(q)$ whose Cayley graphs are the
Diestel-Leader graphs $DL_d(q)$ with respect to a given generating set $S_{d,q}$.  These groups provide a geometric
generalization of the family of lamplighter groups, whose Cayley graphs with respect to a certain generating set are the
Diestel-Leader graphs $DL_2(q)$.  Bartholdi, Neuhauser and Woess in \cite{BNW} show that for $d \geq 3$, $\Gamma_d(q)$ is of type $F_{d-1}$ but not $F_d$.  We show below that these groups have dead end elements of arbitrary depth with respect to the generating set $S_{d,q}$, as well as infinitely many cone types and hence no regular language of geodesics.  These results are proven using a combinatorial formula to compute the word length of group elements with respect to $S_{d,q}$ which is also proven in the paper and relies on the geometry of the Diestel-Leader graphs.
\end{abstract}

\maketitle

\section{Introduction}

We investigate the metric properties of a family of groups whose Cayley graphs with respect to a carefully chosen generating
set are the Diestel-Leader graphs $\dll$. These graphs are subsets of a product of $d$ infinite trees of valence $q+1$.  We call these
groups {\em Diestel-Leader groups} and denote them $\GG$.   More general Diestel-Leader graphs were introduced in \cite{DL}
as a potential answer to the question ``Is any connected, locally finite, vertex transitive graph quasi-isometric to the Cayley graph of a
finitely generated group?"  It was first shown in \cite{EFW} that $DL_2(m,n)$, the Diestel-Leader graph which is a subset of
a product of two trees of valence $m+1$ and $n+1$ respectively, is not quasi-isometric to the Cayley graph of any finitely generated group when $m \neq n$.  It is
proven in \cite{BNW} that Diestel-Leader graphs which are subsets of the product of any
number of trees of differing valence are not Cayley graphs of finitely generated groups.

It is well known that the Cayley graph of the wreath product $L_n=\Z_n \wr \Z$, often called the {\em lamplighter group},
with respect to the generating set $\{ t,ta, ta^2, \dots, ta^{n-1}\}$ (where $a$ is the generator of $\Z_n$ and $t$ generates $\Z$) is the
Diestel-Leader graph $DL_2(n)$. This graph is a subset of the product of two trees of constant valence $n+1$. The groups we
study below provide a geometric generalization of the family of lamplighter groups, as their Cayley graphs generalize the
geometry of the lamplighter groups, that is, their Cayley graphs with respect to a natural generating set $S_{d,q}$ are the
``larger" Diestel-Leader graphs $DL_d(q)$, which are subsets of the product of $d$ trees of constant valence $q+1$, and are
defined explicitly in Section \ref{sec:DLgraphs} below.

Bartholdi, Neuhauser and Woess in \cite{BNW} present a construction of a group which we denote $\GG$, a generating set
$S_{d,q}$ and an identification with the graph $DL_d(q)$ which they prove to be the Cayley graph $\Gamma(\GG,S_{d,q})$.
Moreover, they provide a simple metric criterion for when their construction holds, namely either $d=2$, $d=3$ or if $d \geq 4$ and
$q = p_1^{e_1}p_2^{e_2} \cdots p_r^{e_r}$ is the prime power decomposition of $q$, then $p_i > d-1$ for all $i$.  They show
that the groups $\GG$ are type $F_{d-1}$ but not $F_d$ when $d \geq 3$, hence not automatic.  We note that there are still
open cases where it is not known whether $DL_d(q)$ is the Cayley graph of a finitely generated group; the smallest open case
is $DL_4(2)$.

Random walks in the Cayley graph $\Gamma(\GG,S_{d,q})$ are studied in \cite{BNW}, and a presentation for the group is given
explicitly.  For example, when $d=3$ Bartholdi, Neuhauser and Woess obtain the presentation
$$\Gamma_3(m) \cong \langle a,s,t | a^m=1, \ [a,a^t] = 1, \ [s,t] = 1, \ a^s = aa^t \rangle.$$
When $m=p$ is prime, it is shown in \cite{CT} that $\Gamma_3(p)$ is a cocompact lattice in $Sol_5 \left({\mathbb F}_p((t))
\right)$, and that its Dehn function is quadratic.  The Dehn function of $\Gamma_3(m)$ is studied for any $m$  in \cite{KR}
where it is shown to be at most quartic.  It was mentioned to the authors by Kevin Wortman that arguments analogous to those of Gromov in \cite{G} imply that  the Dehn function
of $\GG$ is quadratic regardless of the values of $d \geq 3$ and $q$.  When the relation $a^m=1$ is removed from the
presentation above, one obtains Baumslag's metabelian group $\Gamma$ which, in contrast to $\Gamma_3(m)$, has exponential
Dehn function \cite{KR}.  Baumslag defined this group to provide the first example of a finitely presented group with an
abelian normal subgroup of infinite rank.

It is noted in \cite{BNW} that $\GG$ is in most cases an automata group, hence a self-similar group.  Metric properties of
self-similar groups are in general not well understood.  In this paper we seek to answer many of the standard geometric
group-theoretic questions related to metric properties of groups and their Cayley graphs for these Diestel-Leader groups
$\GG$.  Such properties often rely on the ability to compute word length of elements within the group; we begin by proving
that a particular combinatorial formula yields the word length of elements of $\GG$ with respect to the generating set
$S_{d,q}$.  This formula relies on the symmetry present in the Diestel-Leader graph, and we subsequently use it to prove
that $\GG$ has dead end elements of arbitrary depth with respect to $S_{d,q}$.  This generalizes a result of Cleary and
Riley \cite{CR1,CR2} which proves that $\Gamma_3(2)$ with respect to a generating set similar to $S_{3,2}$ has dead end
elements of arbitrary depth, the first example of a finitely presented group with this property.  The word length formula is
used in later sections to show that $\GG$ has infinitely many cone types, and hence no regular language of geodesics with
respect to $S_{d,q}$.

\section{Definitions and Background on Diestel-Leader graphs}
\label{sec:DLgraphs}

To define $DL_d(q)$, let $T$ be a homogeneous, locally finite, connected tree in which the degree of each vertex is $q+1$.
This tree has an orientation such that each vertex $v$ has a unique predecessor $v^-$ and $q$ successors $w_1,w_2, \cdots
,w_q$ such that $w_i^-=v$ for $1 \leq i \leq q$. The transitive closure of the set of relationships of the form $v^- < v$
induces the partial order $\preccurlyeq$. In this partial order, any two vertices $v, w \in T$ have a greatest common
ancestor $v\curlywedge w$. Choose a basepoint $o \in T$, and define a height function $h(v)= d(v, o \curlywedge v)-d(o,
o\curlywedge v)$, where $d(x,y)$ denotes the number of edges on the unique path in $T$ from $x$ to $y$. With this
definition, note that $h(v^-) = h(v) - 1$.

Let $T_1$, $T_2, \cdots ,T_d$ denote $d$ copies the tree $T$, with basepoints $o_i$ and height functions $h_i$ for $1 \leq i
\leq d$. The Diestel-Leader graph $DL_d(q)$ is the graph whose vertex set $V_d(q)$ is the set of $d$-tuples $(x_1, x_2,
\dots , x_d)$ where $x_i$ is a vertex of $T_i$ for each $i$, and $h_1(x_1)+ \cdots +h_d(x_d)=0$. Two vertices $x=(x_1,
\dots, x_d)$ and $y=(y_1, \dots, y_d)$ are connected by an edge if and only if there are two indices $i$ and $j$, with $i
\neq j$, such that $x_i$ and $y_i$ are connected by an edge in $T_i$, $x_j$ and $y_j$ are connected by an edge in $T_j$, and
$x_k=y_k$ for $k \neq i,j$.

There is a projection $\Pi: V_d(q) \rightarrow (\Z^2)^d$ given by $$\Pi(x)=\Pi(x_1, x_2, \dots, x_d)=\mlq$$ where
$m_i=d(o_i,o_i\curlywedge x_i)$ and $l_i=d(x_i,o_i\curlywedge x_i)$. In particular, $0 \leq m_i$ and $0 \leq l_i$ for all
$i$.  Note that in $T_i$, the shortest path from $o_i$ to $x_i$ has length $m_i+l_i$, and recall that $h_i(x_i)=l_i-m_i$.
The defining conditions of the Diestel-Leader graph ensure that $\sum_{i=1}^d l_i-m_i = 0$.

In \cite{BNW} it is shown that these graphs are Cayley graphs of certain matrix groups, when a simple metric condition is
satisfied. Specifically, let $\mathcal{L}_q$ be a commutative ring of order $q$ with multiplicative
unit 1, and suppose $\mathcal{L}_q$ contains distinct elements $l_1, \dots, l_{d-1}$ such that if $d\geq 3$, their pairwise differences are invertible.  Define a ring of polynomials in the formal variables $t$ and $(t+l_i)^{-1}$ for $1 \leq i \leq d-1$ with
finitely many nonzero coefficients lying in ${\mathcal L}_q$:
$${\mathcal R}_d({\mathcal L}_q) = {\mathcal L}_q[t,(t+l_1)^{-1},(t+l_2)^{-1}, \cdots ,(t+l_{d-1})^{-1}].$$

It is proven in \cite{BNW} that the group $\Gamma_d(\mathcal{L}_q)$ (which we denote by $\GG$) of affine matrices of the
form
$$\left( \begin{array}{cc} (t+l_1)^{k_1} \cdots (t+l_{d-1})^{k_{d-1}} & P \\ 0 & 1 \end{array} \right), \text{ with }
k_1,k_2, \cdots ,k_{d-1} \in \Z \text{ and }P \in {\mathcal R}_d({\mathcal L}_q)$$
has Cayley graph $\dll$ with respect to the generating set $S_{d,q}$ consisting of the matrices
$$\left( \begin{array}{cc} t+l_i & b \\ 0 & 1 \end{array} \right)^{\pm 1}, \text{ with } b \in {\mathcal L}_q, \ i \in
\{1,2, \cdots ,d-1\} \text{ and }$$
$$ \left( \begin{array}{cc} (t+l_i)(t+l_j)^{-1} & -b(t+l_j)^{-1} \\ 0 & 1 \end{array} \right), \text{ with } b \in {\mathcal
L}_q, \ i,j \in \{1,2, \cdots ,d-1\}, \ i \neq j.$$
This construction holds for any value of $q$ when $d=2$ or $d=3$, and when $d \geq 4$ and $q = p_1^{e_1}p_2^{e_2} \cdots p_r^{e_r}$
is the prime power decomposition of $q$, we require that $p_i > d-1$ for all $i$.  We refer the reader to \cite{BNW} for
more details on this construction and the identification between the group and the Cayley graph $DL_d(q)$.

In exploring the metric properties of the groups $\GG$, and hence the Cayley graphs $DL_d(q)$, one often needs to keep track
of {\em edge types} along a path in $DL_d(q)$ rather than the specific generators which label the edges along the path.
Given any vertex $x=(x_1, x_2, \dots, x_d)$, by an edge of type ${\textbf e_i}-{\textbf e_j}$ emanating from vertex $x$ we
mean an edge with one endpoint at $x$ and the other at $y=(y_1, y_2, \dots , y_d)$ where $y_k=x_k$ for $k \notin \{i,j\}$,
$y_i^-=x_i$, and $y_j=x_j^-$. Note that $h_i(y_i)=h_i(x_i)+1$ and $h_j(y_j)=h_j(x_j)-1$. There are exactly $q$ possible
choices for $y_i$, so there are $q$ distinct edges of type ${\textbf e_i}- {\textbf e_j}$ emanating from $x$.

Since the vertices of $DL_d(q)$ are identified with the elements of $\GG$, we abuse notation and consider the projection map
$\Pi$ to be a map from the group $\GG$ to $(\Z_2)^d$, and write
$$\Pi(g) =\Pi(x=(x_1, x_2, \dots, x_d)) =\mlg $$ when $g \in \GG$ is identified with the vertex $x$ in $\dll$. We remark that the basepoint vertex $o=(o_1, \ldots, l_d)$ in $DL_d(q)$ is identified with the identity element in $\GG$.

\section{Computing Word length in $\GG$ with respect to $S_{d,q}$}
\label{sec:wordlength}

Let $S_{d,q}$ be the generating set for $\Gamma_d(q)$ so that the Cayley graph $\Gamma(\GG,S_{d,q})$ is $DL_d(q)$. We will
show that the word length of an element with respect to $S_{d,q}$ depends only on $\Pi(g)$, and not on $g$ itself. In the course of establishing the formula for word length, it is often sufficient to keep track of the edge types along a path, rather than the edge labels
themselves. Given a vertex $v \in DL_d(q)$, we have defined edges of type ${\textbf e_i}- {\textbf e_j}$ emanating from $v$.
By a path of type $\alpha_1 \alpha _2 \cdots \alpha_r$ starting at $v$, where $\alpha_k=({\textbf e_{i_k}}-{\textbf
e_{j_k}})^{p_k}$, with $p_k\geq 0$ for each $k$, we mean a path beginning at $v$ which follows $p_1$ edges of type ${\textbf
e_{i_1}}-{\textbf e_{j_1}}$, then $p_2$ edges of type ${\textbf e_{i_2}}-{\textbf e_{j_2}}$, and so on.

 We begin by defining a function $f$ from $\Gamma_d(q)$ to the natural numbers, a candidate for the word length function $l: \GG \rightarrow \mathbb{N}$ for elements of $\GG$ with respect to the generating set $S_{d,q}$.

\begin{definition}\label{thm:wordlength}
Let $g \in \Gamma_d(q)$, with $\Pi(g)= \mlg$, and $\sigma$ in $\Sigma_d$, the symmetric group on $d$ letters.  Define
\begin{itemize}
\item $A_{\sigma(d)}(g)= \sum_{j=1}^d m_{\sigma(j)}(g)$ and
    $A_{\sigma(i)}(g)=\sum_{j=2}^{i}m_{\sigma(j)}(g)+\sum_{k=i}^{d-1}l_{\sigma(k)}(g)$ for $2 \leq i \leq d-1$.

    \smallskip

\item $f_{\sigma}(g,i)=m_{\sigma(1)}(g)+l_{\sigma(d)}(g)+ A_{\sigma(i)}(g)$ for $2 \leq i \leq d$.

    \smallskip

\item $f_{\sigma}(g)= \max_{2 \leq i \leq d} f_{\sigma}(g,i)$.

    \smallskip

\item $f(g)= \min_{\sigma \in \Sigma_d}f_{\sigma}(g)$.

\end{itemize}

\end{definition}

In order to establish that the function $f$ defined above is the word length function, we use the following general lemma.

\begin{lemma}\label{lemma:length}
Given a group $G$ with generating set $S$, let $l: G \rightarrow \N$ be the word length with respect to $S$. If $f:G
\rightarrow \N$ is another function satisfying
\begin{enumerate}
\item $f(g)=0$ if and only if $g$ is the identity element, \item For every $g \in G$, $l(g) \geq f(g)$, \item For every
    $g \in G$, there exists some $s \in S$ with $f(gs)=f(g)-1$,
\end{enumerate}
then $l(g)=f(g)$ for every $g \in G$.
\end{lemma}

\begin{proof}
Let $g \in G$, and suppose $f(g)=n$. Then by property (3) there exist $s_1, s_2, \dots, s_n \in S$ satisfying $f(gs_1s_2
\cdots s_n)=0$. By property (1), $g=s_n^{-1} \cdots s_2^{-1}s_1^{-1}$, so $l(g) \leq f(g)$. Hence by property (2) we have
$l(g)=f(g)$.
\end{proof}

Clearly, for the function $f$ defined in Definition \ref{thm:wordlength} we have $f(g)=0$ if and only if $g$ is the identity element. The other two properties of the function $f$ will be
verified in Propositions \ref{prop2} and \ref{prop:decrease} below. It then follows from Lemma \ref{lemma:length} that the function $f$ defined in Definition \ref{thm:wordlength} is the word length function for $\GG$ with respect to the generating set $S_{d,q}$.

\begin{proposition}\label{prop2}
Let $g \in \Gamma_d(q)$ with $\Pi(g)= \mlg$, let $f(g)$ be as in Definition \ref{thm:wordlength} and let $l(g)$ be the word length of $g$ with respect to the generating set $S_{d,q}$. Then $l(g) \geq f(g)$.
\end{proposition}

\begin{proof}

Let $\gamma$ be a path of length $n$ in $DL_d(q)$ from $o$ to the vertex $x$
identified with $g$, thus $\gamma$ corresponds naturally to a word $a_1a_2a_3 \cdots a_n$ with $a_i \in S_{d,q}$ for $1 \leq i \leq n$. We will show that for some choice of $\sigma \in \Sigma_d$
we have $n \geq f_{\sigma}(g,i)$  for every $2 \leq i \leq d$. It follows that $n \geq f_{\sigma}(g) \geq f(g)$, and thus $l(g) \geq f(g)$.

We begin by choosing the permutation $\sigma \in \Sigma_d$. Along the path $\gamma$ from $o$ to $x$, there must be points
where the $k^{th}$ coordinate is $y_k=o_k \curlywedge x_k $ for $1 \leq k \leq d$. Let $v^1$ be the first such point, so
$v^1_{i_1}=y_{i_1}$ for some $i_1$ with $1 \leq i_1 \leq d$. By the definition of $v^1$, we know that $v^1_k \curlywedge x_k
=y_k$ for $k \neq i_1$. Thus, on the portion on the path from $v^1$ to $x$, there must be points where the $k^{th}$
coordinate is $y_k=o_k \curlywedge x_k $ for each $1 \leq k \leq d$, $k \neq i_1$. Let $v^2$ be the first such point, so
$v^2_{i_2}=y_{i_2}$ for some $i_2$ with $1 \leq i_2 \leq d$, $i_2 \neq i_1$. Continuing in this manner, we define points
$v^1,v^2, \dots, v^d$, each with a distinct associated coordinate $i_1, i_2, \dots i_d$ such that the $i_k^{th}$ coordinate
of $v^k$ is $y_{i_k}$.  Let $\sigma \in \Sigma_d$ be the unique permutation defined by $\sigma(k)=i_k$ for $1 \leq k \leq
d$.

First we consider the point $v^j$, for $2 \leq j \leq d-1$, and suppose that the prefix $a_1 \dots
a_r$ corresponds to the subpath of $\gamma$ starting at $o$ and ending at $v^j$. Then for every $p$ with $1 \leq p \leq j$ the path $a_1 \dots  a_r$ must contain at least
$m_{\sigma(p)}(g)$ edges of type ${\textbf e_t}-{\textbf e_{\sigma(p)}}$, where $t \neq {\sigma(p)}$ may vary by edge, so $r
\geq \sum_{p=1}^{j}m_{\sigma(p)}(g)$. However, for every $p$ with $j \leq p \leq d$, the path $a_{r+1} \dots a_n$ must
contain at least $l_{\sigma(p)}(g)$ edges of type ${\textbf e_{\sigma(p)}}-{\textbf e_t}$, where $t \neq {\sigma(p)}$ may
vary by edge, so $n-r \geq \sum_{p=j}^{d}l_{\sigma(p)}(g)$. Thus,
\begin{align*}n=r+(n-r) &\geq \sum_{p=1}^{j}m_{\sigma(p)}(g) + \sum_{p=j}^{d}l_{\sigma(p)}(g)\\
&=m_{\sigma(1)}(g)+A_{\sigma(j)}(g) +l_{\sigma(d)}(g)\\ &=f_{\sigma}(g,j)
\end{align*} for every $2 \leq j \leq d-1$.

For the case $j=d$, we use a slightly different argument. In this case, let  $a_1 \dots a_r$ be the path from $o$ to $v^1$,
and let $a_{r+1} \dots a_s$ be the path from $v^1$ to $v^d$. Then  the path $a_1 \dots a_r$ must contain at least
$m_{\sigma(1)}(g)$ edges of type ${\textbf e_t}-{\textbf e_{\sigma(1)}}$, so $r \geq m_{\sigma(1)}(g)$. Similarly, the path
$a_{s+1} \dots a_n$ must contain at least $l_{\sigma(d)}(g)$ edges of type ${\textbf e_{\sigma(d)}}-{\textbf e_t}$, so $n-s
\geq l_{\sigma(d)}(g)$.

For each $p \neq 1$,  $y_{\sigma(p)} \curlywedge v^1_{\sigma(p)} = y_{\sigma(p)}$, so for each such $p$, there must by at
least $h_{\sigma(p)}(v^1_{\sigma(p)})- h_{\sigma(p)}( y_{\sigma(p)})$ letters corresponding to generators of type ${\textbf
e_t}-{\textbf e_{\sigma(p)}}$ for various choices of $t$ in the word $a_{r+1} \cdots a_s$. Thus, $s-r \geq \sum_{p=2}^d
h_{\sigma(p)}(v^1_{\sigma(p)})- h_{\sigma(p)}(y_{\sigma(p)})$. Now since $\sum_{p=1}^d h_{\sigma(p)}(v_{\sigma(p)}^1) = 0$
and $h_{\sigma(1)}(v^1_{\sigma(1)})=-m_{\sigma(1)}(g)$, we must have $\sum_{p=2}^d h_{\sigma(p)}(v^1_{\sigma(p)}) =
-h_{\sigma(1)}(v^1_{\sigma(1)})=m_{\sigma(1)}(g)$. Furthermore, $h_{\sigma(p)}(y_{\sigma(p)})=-m_{\sigma(p)}(g)$ for every $2 \leq p \leq d$. Hence,

  \begin{align*}s-r &\geq \sum_{p=2}^d \left( h_{\sigma(p)}(v^1_{\sigma(p)})- h_{\sigma(p)}(y_{\sigma(p)}) \right) \\
  &= m_{\sigma(1)}(g)- \sum_{p=2}^d h_{\sigma(p)}(y_{\sigma(p)})\\
  &= \sum_{p=1}^d m_{\sigma(p)}(g) = A_{\sigma(d)}(g).
  \end{align*}
 Thus we have

 \begin{align*} n &= r + (s-r) + (n-s) \\
 &\geq m_{\sigma(1)}(g)+ A_{\sigma(d)}(g)+l_{\sigma(d)}(g)\\
 &=f_{\sigma}(g,d).
 \end{align*}

 Hence, we have shown that $n \geq f_{\sigma}(g,j)$  for every $2 \leq j \leq d$, as desired.

\end{proof}

To complete the argument, we must prove that $f$ satisfies the third and final property of Lemma \ref{lemma:length}. In
doing so, it is often necessary to keep track of which values of $l_{\chi(i)}(g)$ in $\Pi(g)$ are zero for a given
permutation $\chi$, so we first prove several preliminary lemmas.

\begin{lemma}\label{lemma:index-set}
Let $\chi \in \Sigma_d$ be any permutation and $g \in \GG$ any nontrivial element.
\begin{enumerate}
\item If $l_{\chi(d)}(g)=0$, let $n$ be the maximal value of $j$ with $1 \leq j \leq d-1$ so that  $l_{\chi(j)}(g) \neq
    0$.
 Then
$$\max_{2 \leq i \leq d} A_{\chi(i)}(g) = \max_{2 \leq i \leq n, i=d} A_{\chi(i)}(g)$$
\item If $l_{\chi(1)}(g)=0$, let $k$ be the minimum value of $j$ with $2 \leq j \leq d$ so that $l_{\chi(j)}(g) \neq 0$.
    Then
$$\max_{2 \leq i \leq d} A_{\chi(i)}(g) = \max_{k \leq i \leq d} A_{\chi(i)}(g).$$
\end{enumerate}
\end{lemma}

\begin{proof}
Since $g$ is nontrivial, the values of $n$ and $k$ defined in the statement above both exist. The proof of (1) follows from the fact that if $l_{\chi(n+1)}(g) = l_{\chi(n+2)}(g) = \cdots = l_{\chi(d-1)}(g) = 0$, then for
$n+1 \leq i \leq d-1$ we have $A_{\chi(i)}(g) \leq A_{\chi(d)}(g)$. Similarly, to prove (2), if $l_{\chi(2)}(g)= \cdots =
l_{\chi(k-1)}(g)=0$ for $2 \leq i \leq k-1$ we have $A_{\chi(i)}(g) \leq A_{\chi(k)}(g)$.
\end{proof}

\begin{lemma}\label{prop:main}
Fix $g \in \GG$.  Let $\sigma \in
\Sigma_d$ with $l_{\sigma(1)}(g) = 0$.  Let $\tau \in \Sigma_d$ be defined by $\tau(i) = \sigma(i+1)$ for $1 \leq i <d$ and
$\tau(d) = \sigma(1)$.  Then $f_{\sigma}(g) \geq f_{\tau}(g)$.
\end{lemma}

\begin{proof}

First note that for $2 \leq i \leq d-2$ we have
$$A_{\sigma(i+1)}(g) = m_{\sigma(2)}(g) + \cdots + m_{\sigma(i+1)}(g) + l_{\sigma(i+1)}(g) + \cdots  + l_{\sigma(d-1)}(g)$$
and
\begin{align*} \ti(g) &= m_{\tau(2)}(g) + \cdots + m_{\tau(i)}(g) + l_{\tau(i)}(g) + \cdots  l_{\tau(d-1)}(g) \\
 &= m_{\sigma(3)}(g) + \cdots + m_{\sigma(i+1)}(g) + l_{\sigma(i+1)}(g) + \cdots + l_{\sigma(d)}(g).
 \end{align*}
Hence for $2 \leq i \leq d-2$ we have $$A_{\sigma(i+1)}(g) = \ti(g) +m_{\sigma(2)}(g)-l_{\sigma(d)}(g).$$

The lemma is clearly true if $g$ is the identity element, so we may assume for the remainder of the proof that $g$ is nontrivial. Using the definition of $k$ given in Lemma \ref{lemma:index-set} we have $\max_{2 \leq i \leq d} A_{\sigma(i)}(g) = \max_{k \leq i \leq d}
A_{\sigma(i)}(g)$, we may assume that $f_{\sigma}(g) = f_{\sigma}(g,i)$ for $k \leq i \leq d$, that is, $f_{\sigma}(g) =
m_{\sigma(1)}(g) + l_{\sigma(d)}(g) + \si(g)$ for some $i$ with $k \leq i \leq d$. We must show for each $j$ with $2 \leq j
\leq d$ that $f_{\sigma}(g,i) \geq f_{\tau}(g,j)$.  From this it follows that $f_{\sigma}(g) \geq f_{\tau}(g)$ as desired.
We consider three subcases, as follows.

\textbf{Case 1.} Suppose that $2 \leq j \leq d-2$.  Using the formula above, we see that
\begin{align*}
f_{\sigma}(g,i)&= m_{\sigma(1)}(g) + l_{\sigma(d)}(g) + A_{\sigma(i)}(g)\\
 &\geq m_{\sigma(1)}(g) + l_{\sigma(d)}(g) + A_{\sigma(j+1)}(g) \\
 &= m_{\sigma(1)}(g) + l_{\sigma(d)}(g) + A_{\tau(j)}(g) + m_{\sigma(2)}(g) - l_{\sigma(d)}(g) \\
 & \geq m_{\sigma(2)}(g) + A_{\tau(j)}(g) = m_{\tau(1)}(g) + l_{\tau(d)}(g) + A_{\tau(j)}(g)\\
 & = f_{\tau}(g,j)
\end{align*}
where the first inequality holds because $f_{\sigma}(g)=f_{\sigma}(g,i)$ and thus $A_{\sigma(i)}(g) \geq A_{\sigma(j)}(g)$
for $i \neq j$, and the penultimate equality holds because $l_{\tau(d)}(g) = l_{\sigma(1)}(g) =0$.

\textbf{Case 2.} Suppose that $j=d$.  Then,
\begin{align*}
f_{\sigma}(g,i) &=  m_{\sigma(1)}(g) + l_{\sigma(d)}(g) + A_{\sigma(i)} (g)\\ &\geq m_{\sigma(1)}(g) + l_{\sigma(d)}(g) +
A_{\sigma(k)}(g) \\ &= m_{\sigma(1)}(g) + m_{\sigma(2)}(g) + \cdots + m_{\sigma(k)}(g) + l_{\sigma(k)}(g) +
l_{\sigma(k+1)}(g) + \cdots + l_{\sigma(d)}(g) \\ &\geq m_{\sigma(2)}(g) + 0 + l_{\sigma(k)}(g) + l_{\sigma(k+1)}(g) +
\cdots + l_{\sigma(d)}(g)  \\ &=  m_{\tau(1)}(g) + l_{\tau(d)}(g) + \sum_{r=1}^d m_{\tau(r)(g)}  = f_{\tau}(g,d)
 \end{align*}
where the last line relies on the fact thats that $l_{\tau(d)}(g) = l_{\sigma(1)}(g) = 0$ and by our choice of $k$,
$$\sum_{r=1}^d m_{\tau(r)}(g) = \sum_{r=1}^d m_{\sigma(r)}(g) = \sum_{r=1}^d l_{\sigma(r)}(g) = \sum_{r=k}^d
l_{\sigma(r)}(g).$$

\textbf{Case 3. } When $j=d-1$ we differentiate between $2 \leq i \leq d-1$ and $i=d$.  Recall that $f_{\sigma}(g) =
f_{\sigma}(g,i)$.

First let $2 \leq i \leq d-1$ and recall that $l_{\tau(d)}(g) = l_{\sigma(1)}(g) = 0$ by assumption. In this case,
\begin{align*}
f_{\tau}(g,d-1) &= m_{\tau(1)}(g) + m_{\tau(2)}(g) + \cdots + m_{\tau(d-1)}(g) + l_{\tau(d-1)}(g) + l_{\tau(d)}(g) \\ &=
m_{\sigma(2)}(g) + \cdots + m_{\sigma(d)}(g) + l_{\sigma(d)}(g).
\end{align*}
Additionally, we are assuming that $A_{\sigma(i)} \geq A_{\sigma(d)}$.  Writing out this inequality and canceling identical
terms from both sides of the inequality yields
$$l_{\sigma(i)}(g) + l_{\sigma(i+1)}(g) + \cdots l_{\sigma(d-1)}(g) \geq m_{\sigma(1)}(g) + m_{\sigma(i+1)}(g)  +
m_{\sigma(i+2)}(g)+ \cdots + m_{\sigma(d)}(g).$$
Hence,
\begin{align*}
f_{\sigma}(g,i) &= m_{\sigma(1)}(g) + m_{\sigma(2)}(g) + \cdots + m_{\sigma(i)}(g) + l_{\sigma(i)}(g) + \cdots +
l_{\sigma(d)}(g)\\
 &\geq \left(m_{\sigma(1)}(g)+ \cdots + m_{\sigma(i)}(g)\right) + \left(m_{\sigma(1)}(g) + m_{\sigma(i+1)}(g)  +
 m_{\sigma(i+2)}(g)+ \cdots + m_{\sigma(d)}(g)\right) + l_{\sigma(d)} (g) \\
  &\geq m_{\sigma(2)}(g) + \cdots + m_{\sigma(d)}(g) + l_{\sigma(d)}(g) \\
  &= m_{\tau(1)}(g) + m_{\tau(2)}(g) + \cdots + m_{\tau(d-1)}(g) + l_{\tau(d-1)}(g)+l_{\tau(d)}(g) = f_{\tau}(g,d-1).
\end{align*}

Now assume that $i=d$.  In this case,
\begin{align*}
f_{\sigma}(g,d) &= m_{\sigma(1)}(g) + l_{\sigma(d)}(g) + \sum_{r=1}^d m_{\sigma(i)}(g) \\
&\geq m_{\sigma(2)}(g) + \cdots +m_{\sigma(d)}(g) + l_{\sigma(d)}(g)\\
&= m_{\tau(1)}(g)  + m_{\tau(2)}(g) + \cdots + m_{\tau(d-1)}(g) +
l_{\tau(d-1)}(g) \\
&= m_{\tau(1)}(g)  + m_{\tau(2)}(g) + \cdots + m_{\tau(d-1)}(g) +
l_{\tau(d-1)}(g) + l_{\tau(d)}(g) = f_{\tau}(g,d-1) .
\end{align*}
\end{proof}

 If $g \in \Gamma_d(q)$ is nontrivial,  let $\Theta_g = \{ \sigma \in \Sigma_d| f(g) = f_{\sigma}(g) \}$, and  let $\Theta'_g =
 \{\sigma \in \Theta_g | l_{\sigma(1)}(g) \neq 0 \}$. Then Lemma \ref{prop:main} has the following corollary.

\begin{corollary}\label{cor:main}
If $g \in \GG$ is not the identity element, then $\Theta_g'$ is not empty, that is, there exists $\sigma \in \Sigma_d$ such
that $f(g) = f_{\sigma}(g)$ and $l_{\sigma(1)}(g) \neq 0$.
\end{corollary}

\begin{proof} Suppose that $\chi \in \Theta_{g}$, and $l_{\chi(1)}(g) =0$.  Let $k$ be defined as in Lemma \ref{lemma:index-set}.
Applying Lemma \ref{prop:main} $k-1$ times, we obtain the corollary.

\end{proof}

\begin{lemma}\label{lemma:twozeros}
Let $g \in \GG$ and $\sigma \in \Sigma_d$ with $m_{\sigma(d)}(g)=l_{\sigma(d)}(g)=0$. Define $\tau \in \Sigma_d$ such that
$\tau(i)=\sigma(i-1)$ for $i \geq 2$, and $\tau(1)=\sigma(d)$. Then $f_{\sigma}(g)=f_{\tau}(g)$.
\end{lemma}

\begin{proof}
One directly verifies, using arguments as in the previous lemma, that $f_{\sigma}(g,i)=f_{\tau}(g, i+1)$ for $2 \leq i \leq
d-1$, and $f_{\sigma}(g,d)=f_{\tau}(g,2)$.
\end{proof}

We immediately obtain the following corollary.

\begin{corollary}\label{cor:twozeroes}
If $g \in \GG$ is not the identity element, there exists $\sigma \in \Theta_g$ such that either $l_{\sigma(d)}(g)\neq 0$ or
$ m_{\sigma(d)}(g)\neq 0$.
\end{corollary}

The next proposition uses the above lemmas and corollaries to prove that the function $f$ satisfies the third property of
Lemma \ref{lemma:length}.

\begin{proposition}\label{prop:decrease}
Let $g \in \Gamma_d(q)$ be a nontrivial group element, and let $f(g)$ be as in Definition \ref{thm:wordlength}. Then there
exists $s \in S_{d,q}$ with $f(gs)=f(g)-1$.
\end{proposition}

\begin{proof}
If $g \in S_{d,q}$, that is, $g$ is a generator of $\Gamma_d(q)$, then it is easy to see that $f(g) = 1$ and choosing $s =
g^{-1}$, that $f(gs) = 0$ and the condition of the proposition is satisfied.  From now on we assume that $g \notin S_{d,q}$
and hence for any $s \in S_{d,q}$ we know that $gs$ is nontrivial.

Let $x$ be the vertex in $DL_d(q)$ identified with $g$; recall that we write $\Pi(g)$ for $\Pi(x)$.

\textbf{Case 1.} There exists $\sigma \in \Theta_g$ with $l_{\sigma(d)}(g) \neq 0$. If in addition $l_{\sigma(1)}(g) > 0$,
or $l_{\sigma(1)}(g)=m_{\sigma(1)}(g)=0$, choose $s$ to be any generator corresponding to an edge of type ${\textbf
e_{\sigma(1)}}-{\textbf e_{\sigma(d)}}$. If $l_{\sigma(1)}(g) = 0$ and $m_{\sigma(1)}(g) >0$, let $w$ be the vertex in
$T_{\sigma(1)}$ adjacent to $x_{\sigma(1)}$ on the unique shortest path from $o_{\sigma(1)}$ to $x_{\sigma(1)}$. Choose $s$
to be any generator of type ${\textbf e_{\sigma(1)}}-{\textbf e_{\sigma(d)}}$ so that if $z$ is the vertex in $DL_d(q)$
identified with $gs$, then $z_{\sigma(1)} \neq w$. Then we have:
\begin{enumerate}
\item $(m_{\sigma(d)}(gs),l_{\sigma(d)}(gs)) = (m_{\sigma(d)}(g),l_{\sigma(d)}(g)-1)$, \item
    $(m_{\sigma(1)}(gs),l_{\sigma(1)}(gs)) = (m_{\sigma(1)}(g),l_{\sigma(1)}(g)+1)$, and \item
    $(m_{\sigma(i)}(gs),l_{\sigma(i)}(gs)) = (m_{\sigma(i)}(g),l_{\sigma(i)}(g))$ for $i \neq 1,d$.
\end{enumerate}

But this implies that $A_{\sigma(i)}(gs)=A_{\sigma(i)}(g)$ for every $2 \leq i \leq d$, and hence
\begin{align*}f_{\sigma}(gs)&=m_{\sigma(1)}(gs)+l_{\sigma(d)}(gs)+ \max_{2 \leq i \leq d} A_{\sigma(i)}(gs) \\ &=
m_{\sigma(1)}(g)+(l_{\sigma(d)}(g)-1)+ \max_{2 \leq i \leq d} A_{\sigma(i)}(g)\\ &=f_{\sigma}(g)-1.
 \end{align*}

First we note that the inequality $f(g)-1 \geq f(gs)$ is fairly easy to verify, since $f(g) - 1 = f_{\sigma}(g) - 1 =
f_{\sigma}(gs) \geq f_{\tau}(gs)$ for any $\tau \in \Theta_{gs}$.  Hence $f(g) - 1 \geq f(gs)$.

Now for any $\tau \in \Sigma_d$,  $m_{\tau (i)}(gs) = m_{\tau (i)}(g)$ for every $1 \leq i \leq d$, $l_{\tau (i)}(gs) \neq l_{\tau (i)}(g)$ for only two choices of $i$, and in addition, for one of these values, $l_{\tau (i)}(gs) = l_{\tau (i)}(g)-1$ and for the other, $l_{\tau (i)}(gs) = l_{\tau (i)}(g)+1$. Since for any value of $i$, $l_{\tau (i)}(g)$ (respectively $l_{\tau (i)}(gs)$) appears at most in the formula for $f_{\tau}(g,i)$ (respectively $f_{\tau}(gs,i)$) this implies that $f_{\tau}(g)-1 \leq f_{\tau}(gs)$.
Thus for $\tau \in \Theta_{gs}$, $f(gs)=f_{\tau}(gs) \geq f_{\tau}(g)-1 \geq f(g)-1$, so $f(gs) \geq f(g)-1$ as well. Hence, $f(gs)=
f(g)-1$, as desired.

\textbf{Case 2.} For every $\chi \in \Theta_g$, we assume that $l_{\chi(d)}(g)=0$.  Applying Corollary
\ref{cor:twozeroes} we may choose $\sigma \in \Theta_g$ so that $m_{\sigma(d)}(g) \neq 0$.

Let $w$ be the vertex in $T_{\sigma(d)}$ adjacent to $x_{\sigma(d)}$ on the unique shortest path from $o_{\sigma(d)}$ to
$x_{\sigma(d)}$, and let $n$ be as defined in Lemma \ref{lemma:index-set}. Choose the generator $s \in S_{d,q}$ of type ${\textbf e_{\sigma(d)}}-{\textbf e_{\sigma(n)}}$ so that if
$z$ is the vertex in $DL_d(q)$ identified with $gs$, then $z_{\sigma(d)}=w$. Then we have, for the pair $\sigma$
and $s$:
\begin{enumerate}
\item $(m_{\sigma(d)}(gs),l_{\sigma(d)}(gs)) = (m_{\sigma(d)}(g)-1,l_{\sigma(d)}(g))$, where we note that
    $l_{\sigma(d)}(gs) = l_{\sigma(d)}(g) = 0$, \item $(m_{\sigma(n)}(gs),l_{\sigma(n)}(gs)) =
    (m_{\sigma(n)}(g),l_{\sigma(n)}(g)-1)$,  and \item
    $(m_{\sigma(i)}(gs),l_{\sigma(i)}(gs)) = (m_{\sigma(i)}(g),l_{\sigma(i)}(g))$ for $i \neq n,d$.
\end{enumerate}

For the above choice of $\sigma$ and $s$, we claim that $f_{\sigma}(gs) = f_{\sigma}(g)-1$.  Applying Lemma
\ref{lemma:index-set} to $g$ we see that
$$f_{\sigma}(g) = m_{\sigma(1)}(g) +  l_{\sigma(d)}(g) + \max_{2 \leq i \leq n, i=d} A_{\sigma(i)}(g).$$
As the only index for which $l_{\sigma(i)}(gs) \neq l_{\sigma(i)}(g)$ is $i=n$, we again see that for $j>n$ we have
$l_{\sigma(j)}(gs) =0$. Applying Lemma \ref{lemma:index-set} to $gs$, we see that
$$f_{\sigma}(gs) = m_{\sigma(1)}(gs) +  l_{\sigma(d)}(gs) + \max_{2 \leq i \leq n, i=d} A_{\sigma(i)}(gs).$$
It follows from the definition of $s$ that $A_{\sigma(d)}(gs) = A_{\sigma(d)}(g)-1$.  Similarly, for $2 \leq i \leq n$ we
have $A_{\sigma(i)}(gs) = A_{\sigma(i)}(g)-1$ since neither expression contains $m_{\sigma(d)}$ and both contain
$l_{\sigma(n)}$.  Hence,
$$\max_{2 \leq i \leq n, i=d} A_{\sigma(i)}(gs) = \max_{2 \leq i \leq n, i=d} \left( A_{\sigma(i)}(g)-1 \right) = \left(
\max_{2 \leq i \leq n, i=d} A_{\sigma(i)}(g) \right) - 1.$$
Combining the above reasoning, we see that
\begin{align*}
f_{\sigma}(gs) &= m_{\sigma(1)}(gs) +  l_{\sigma(d)}(gs) + \max_{2 \leq i \leq n, i=d} A_{\sigma(i)}(gs) \\
 &=  m_{\sigma(1)}(g) +  l_{\sigma(d)}(g) + \max_{2 \leq i \leq n, i=d} A_{\sigma(i)}(g) -1 = f_{\sigma}(g) -1.
 \end{align*}

Since $f_{\sigma}(g) = f(g)$ and $f(gs) \leq f_{\sigma}(gs)$ it follows immediately from the fact that $f_{\sigma}(gs)=
f_{\sigma}(g) -1$ that $f(gs) \leq f(g)-1$.

To complete the proof of Proposition \ref{prop:decrease} we must show that $f(gs) \geq f(g)-1$.  First note that for any
$\chi \in \Sigma_d$ it follows from the definition of $f$ that $f_{\chi}(gs) \geq f_{\chi}(g)-3$.  We now show that this
inequality can be improved slightly for $\tau \in \Theta'_{gs}$; for such $\tau$ we will show that $f_{\tau}(gs) \geq f_{\tau}(g)-2$.
Suppose to the contrary that $\tau \in \Theta '_{gs}$ and $f_{\tau}(gs)=f_{\tau}(g)-3$.  This can happen in only one way, namely
all three of the following conditions must be met:
\begin{enumerate}
\item $\tau(1) = \sigma(d)$, \item $\tau(d) = \sigma(n)$, and \item $\max_{2 \leq i \leq n, i=d}A_{\tau(i)}(gs) =
    A_{\tau(d)}(gs)$.
\end{enumerate}
Now $\tau \in \Theta '_{gs}$ implies that $l_{\tau(1)}(gs) \neq 0$, but the first condition above requires that
 $l_{\tau(1)}(gs) = l_{\sigma(d)}(gs) = l_{\sigma(d)}(g) = 0$, a contradiction.  Thus, for all $\tau \in \Theta '_{gs}$, we must
 have $f_{\tau}(gs) \geq f_{\tau}(g)-2$.

It follows from Corollary \ref{cor:main} that $\Theta'_{gs}$ is not empty, so we may choose
$\chi \in \Theta'_{gs}$. If $\chi \notin \Theta_g$, then $f_{\chi}(g) > f_{\sigma}(g)$. Thus we have
$f(gs) = f_{\chi}(gs) \geq f_{\chi}(g)-2 > f_{\sigma}(g)-2$, and hence  $f(gs) \geq f_{\sigma}(g)-1=f(g)-1$, as desired.

If, on the other hand,  $\chi \in \Theta _g$, we make the following claim.

{\bf Claim.} If $\chi \in \Theta _g$, there exists $\tau \in \Theta '_{gs}$ with $f_{\tau}(gs) \geq f_{\tau}(g)-1$.

Proposition \ref{prop:decrease} follows immediately from the Claim, as follows. Let $\tau$ be as in the Claim, so that
$f_{\tau}(gs) \geq f_{\tau}(g)-1$. Then $f(gs) = f_{\tau}(gs) \geq f_{\tau}(g)-1 \geq f(g)-1$, and hence $f(gs) \geq f(g)
-1$, as desired.

To prove the claim, if $f_{\chi}(gs) \geq f_{\chi}(g)-1$ then simply let
$\tau = \chi$. If $f_{\chi}(gs) = f_{\chi}(g)-2$, we use $\chi$ to construct $\tau \in \Theta'_{gs}$ so that $f_{\tau}(gs)
\geq f_{\tau}(g) -1$, as follows.

There exist distinct $u,v \in \{1,2, \cdots ,d\}$ so that $\chi(u) = \sigma(d)$ and $\chi(v) = \sigma(n)$.  We now show that
$1 < u < v < d$. To see that $1<u$, observe that $l_{\sigma(d)}(gs) = 0$ but $l_{\chi(1)}(gs) \neq 0$ since $\chi \in
\Theta'_{gs}$, hence $\sigma(d) \neq \chi(1)$, that is, $u \neq 1$.  To see that $v<d$, observe that $l_{\sigma(n)}(g)
=l_{\chi(v)}(g) \neq 0$. Recall that since $\chi \in \Theta _g$, $l_{\chi(d)}(g) = 0$, and hence $v \neq d$.

Finally, we must show that $u<v$. Since $\chi(1) \neq \sigma(d)$, $m_{\chi(1)}(gs) =
m_{\chi(1)}(g)$. Also, since $\chi(v) \neq \chi(d)$,  $l_{\chi(d)}(gs) = l_{\chi(d)}(g)$.
Thus, in order for $f_{\chi}(gs) = f_{\chi}(g)-2$ it must be the case
that $\max_{2 \leq i \leq d}A_{\chi(i)}(g)-\max_{2 \leq i \leq d}A_{\chi(i)}(gs)=2$.  The only way this can happen is if
$\max_{2 \leq i \leq d}A_{\chi(i)}(g)$ is realized by an expression which contains both $m_{\sigma(d)}(g)$ and $l_{\sigma(n)}(g)$, that is, both $m_{\chi(u)}(g)$ and $l_{\chi(v)}(g)$.  By the construction of the terms $A_{\chi(i)}(g)$ we must have $u<v$ for this to occur, for if $u>v$ and $l_{\chi(v)}(g)$ was a term in the expression which realized $\max_{2 \leq i \leq
d}A_{\chi(i)}(g)$, this expression would also contain $l_{\chi(u)}(g)$, not $m_{\chi(u)}(g)$ as required. Thus $u<v$.

We now construct $\tau \in \Theta'_{gs}$ which satisfies $f_{\tau}(gs) \geq f_{\tau}(g)-1$.  Let $u$ and $v$ be as above,
and define
\begin{itemize}
\item For $i<u$ let $\tau(i) = \chi(i)$. \item For $u \leq i < v$ let $\tau(i) = \chi(i+1)$. \item For $i=v$ let
    $\tau(v) = \chi(u)$. \item For $i>v$ let $\tau(i) = \chi(i)$.
\end{itemize}

We first show that $\tau \in \Theta'_{gs}$. Since $\chi(1) = \tau(1)$ and $\chi(d) = \tau(d)$ we claim that for any
$i$ with $2 \leq i \leq d$ we have $A_{\tau(i)}(gs) \leq \max_{2 \leq j \leq d} A_{\chi(j)}(gs)$.  This is clearly true when
$i=d$, since $A_{\tau(d)}(gs) = A_{\chi(d)}(gs)$.  We consider four remaining cases, and abuse our notation by writing
$m_{\chi(i)}$, $l_{\chi(i)}$, and $A_{\chi(i)}$ instead of $m_{\chi(i)}(gs)$, $l_{\chi(i)}(gs)$, and $A_{\chi(i)}(gs)$,
respectively.

\begin{enumerate}
\item Let $i<u$.  Then
\begin{align*} A_{\chi(i)} &= m_{\chi(2)} + \cdots + m_{\chi(i)} + l_{\chi(i)} + \cdots + l_{\chi(u)} + \cdots +
l_{\chi(v)} + \cdots + l_{\chi(d-1)} \\
& \text{simply rearranging the terms yields}\\
 &= m_{\chi(2)} + \cdots + m_{\chi(i)} + l_{\chi(i)} + \cdots + l_{\chi(u-1)} + l_{\chi(u+1)} +  \cdots +  l_{\chi(v)} +
 l_{\chi(u)} + l_{\chi(v+1)} +  \cdots \\
 & \ \ \ + l_{\chi(d-1)} \\
 & \text{changing to the equivalent indices for } \tau \text{ yields}\\
  &=m_{\tau(2)} + \cdots + m_{\tau(i)} + l_{\tau(i)} + \cdots + l_{\tau(d-1)} \\
  &= A_{\tau(i)}.
\end{align*}

\medskip

\item If $u \leq i <v $ then
\begin{align*} A_{\tau(i)} &= m_{\tau(2)} + \cdots + m_{\tau(u)} + \cdots + m_{\tau(i)} + l_{\tau(i)} + \cdots +
l_{\tau(d-1)} \\
&= m_{\chi(2)} + \cdots + m_{\chi(u-1)} + m_{\chi(u+1)} + \cdots + m_{\chi(i+1)} + l_{\chi(i+1)} + \cdots + l_{\chi(v)}
+ l_{\chi(u)}  \\ & \ \ \ + l_{\chi(v+1)} + \cdots + l_{\chi(d-1)} \\ & \text{adding in the "missing" term
}m_{\chi(u)}=m_{\sigma(d)}>0 \text{ and omitting }l_{\chi(u)}=0 \text{ yields the inequality}\\
 &< m_{\chi(2)} + \cdots + + m_{\chi(u-1)} + m_{\chi(u)} +  m_{\chi(u+1)} + \cdots + m_{\chi(i+1)} + l_{\chi(i+1)} +
 \cdots \\
 & \ \ \ + l_{\chi(v)} + l_{\chi(v+1)} + \cdots + l_{\chi(d-1)} \\
 &= A_{\chi(i+1)}.
\end{align*}

\medskip

\item If $i=v$ then recall that $\tau(v) = \chi(u)=\sigma(d)$.  Note that $\tau(v-1) = \chi(v)$ by the definition of
    $\tau$.  Then
\begin{align*}
A_{\tau(v)} &= m_{\tau(2)} + \cdots + m_{\tau(u-1)} + m_{\tau(u)} + \cdots + m_{\tau(v-1)} + m_{\tau(v)} + l_{\tau(v)} +
\cdots + l_{\tau(d-1)} \\
 &= m_{\chi(2)} + \cdots + m_{\chi(u-1)}+ m_{\chi(u+1)}+ \cdots + m_{\chi(v)} + m_{\chi(u)} + l_{\chi(u)} +
 l_{\chi(v+1)} + \cdots + l_{\chi(d-1)} \\
 & \text{rearranging the existing terms, omitting }l_{\chi(u)}=0 \text{  and adding in the term } \\
  & l_{\chi(v)} \text{ yields the inequality}\\
 &\leq m_{\chi(2)} + \cdots + m_{\chi(u-1)} + m_{\chi(u)} + m_{\chi(v)} + l_{\chi(v)} + \cdots + l_{\chi(d-1)} =
 A_{\chi(v)}.
\end{align*}
\medskip

\item Let $i > v$.  Then
\begin{align*} A_{\chi(i)} &= m_{\chi(2)} + \cdots + m_{\chi(u)} + \cdots + m_{\chi(v)} + \cdots + m_{\chi(i)} +
l_{\chi(i)} + \cdots + l_{\chi(d-1)} \\
& \text{rearranging the existing terms yields }\\
 &= m_{\chi(2)} + \cdots + m_{\chi(u-1)} + m_{\chi(u+1)} + \cdots  + m_{\chi(v)} + m_{\chi(u)} + m_{\chi(v+1)} + \cdots
 + m_{\chi(i)} \\
 & \ \ \ + l_{\chi(i)} + \cdots + l_{\chi(d-1)} \\
  &= A_{\tau(i)}.
\end{align*}

\end{enumerate}

Combining these cases we see that for all $2 \leq i \leq d$ we have $A_{\tau(i)}(gs) \leq \max_{2 \leq j \leq d}
A_{\chi(j)}(gs)$.  Thus
$f_{\tau}(gs) \leq f_{\chi}(gs) = f(gs)$ and hence $f(gs) = f_{\tau}(gs)$, that is, $\tau \in \Theta_{gs}$. Now since $\chi \in \Theta _{gs}'$, this implies that $l_{\chi(1)}(gs) \neq 0$. But $\tau(1)=\chi(1)$, so $l_{\tau(1)}(gs) \neq 0$ as well, and hence $\tau \in \Theta_{gs}'$.

Finally, it remains to show  that $f_{\tau}(gs) = f_{\tau}(g)-1$. Recall from the definition of $\tau$ that $\tau(v) =
\chi(u) = \sigma(d)$ and $\tau(v-1) = \chi(v) = \sigma(n)$.  Additionally, recall from the choice of $s$ that
\begin{enumerate}
\item $(m_{\sigma(d)}(gs),l_{\sigma(d)}(gs)) = (m_{\sigma(d)}(g)-1,l_{\sigma(d)}(g))$, \item
    $(m_{\sigma(n)}(gs),l_{\sigma(n)}(gs)) = (m_{\sigma(n)}(g),l_{\sigma(n)}(g)-1)$, \item
    $(m_{\sigma(i)}(gs),l_{\sigma(i)}(gs)) = (m_{\sigma(i)}(g),l_{\sigma(i)}(g))$ for $i \neq n,d$.
\end{enumerate}
Comparing $A_{\tau(i)}(gs)$ and $A_{\tau(i)}(g)$ for all possible values of $i$ shows that for all $i$, $A_{\tau(i)}(gs) =
A_{\tau(i)}(g)-1$.

From the definition of $\tau$ we see that
$$m_{\tau(1)}(g) = m_{\chi(1)}(g) = m_{\chi(1)}(gs) = m_{\tau(1)}(gs)$$
and
$$l_{\tau(d)}(g) = l_{\chi(d)}(g) = l_{\chi(d)}(gs) = l_{\tau(d)}(gs).$$
Thus $f_{\tau}(gs) = f_{\tau}(g)-1$, which concludes the proof of the Claim, and hence Proposition \ref{prop:decrease}.
\end{proof}

\section{Comparing word length in $\GG$ and distance in the product of trees}

As the Diestel-Leader graph $DL_d(q)$ is a subset of the product of $d$ trees of valence $q+1$, it is natural to compare the word metric on the Cayley graph $DL_d(q)$ to the product metric on the product of trees.  This product metric assigns every edge length one, and simply counts edges in each tree between the coordinates corresponding to two different group elements.  It is a straightforward consequence of the word length formula that these two metrics are quasi-isometric. In Corollary \ref{cor:treedist} below we extend the word length function $f$ to compute the the distance in the word metric (with respect to the generating set $S_d(q)$) between arbitrary group elements.  We conclude with a corollary which constructs a family of quasi-geodesic paths from the vertex corresponding to the identity to that corresponding to any group element.

\begin{theorem}\label{thm:treedist}
Let $l(g)$ denote the word length of $g \in \GG$ with respect to the generating set $S_{d,q}$ and $d_T(g)$ the distance in the product metric on the product of trees between $g$ in $DL_d(q)$ and $\epsilon$, the fixed basepoint corresponding to the identity in $\GG$.  Then
$$\frac{1}{2} d_T(g) \leq l(g) \leq 2 d_T(g)$$
that is, the word length is quasi-isometric to the distance from the identity in the product metric on the product of trees.
\end{theorem}

\begin{proof}
Let $\Pi(g) = \mlq$.  It follows that $d_T(g) = \sum_{i=1}^d m_i + l_i = 2 \sum_{i=1}^d m_i$.  Using the word length formula from Section \ref{sec:wordlength}, we see that for some $\sigma \in \Sigma(d)$,
$$l(g)=f_{\sigma}(g)= \left(m_{\sigma(1)} + l_{\sigma(d)}\right) + \max_{2 \leq i \leq d} A_{\sigma(i)} \leq \left(\sum_{i=1}^d m_i + \sum_{i=1}^d l_i \right) + \sum_{i=1}^d m_i + \sum_{i=1}^d l_i = 2 d_T(g).$$
To obtain a lower bound, note that

\begin{align*}
l(g)= \min _{\sigma \in \Sigma_d} f_{\sigma}(g) &= \min_{\sigma \in \Sigma_d} (m_{\sigma(1)} + l_{\sigma(d)}+ \max _{2 \leq i \leq d} A_{\sigma(i)}(g)) \\
&\geq \min_{\sigma \in \Sigma_d} ( \max _{2 \leq i \leq d} A_{\sigma(i)}(g))
\end{align*}

But for every $\sigma\in \Sigma_d$, $\max _{2 \leq i\leq d} A_{\sigma(i)}(g) \geq A_{\sigma(d)}(g)= \sum_{i=1}^d m_i$, so

$$l(g) \geq \sum_{i=1}^d m_i = \frac{1}{2} d_T(g).$$
Combining these inequalities proves the theorem.
\end{proof}

The first corollary to Theorem \ref{thm:treedist} requires us to extend the techniques of Section \ref{sec:wordlength} in order to compute the distance in the word metric between arbitrary group elements.

\begin{corollary}\label{cor:treedist}
Let $g,h \in \GG$ and let $d_T(g,h)$ denote the distance between the two vertices in $DL_d(q)$ corresponding to $g$ and $h$ with respect to the product metric on the product of trees.  Then
$$\frac{1}{2} d_T(g,h) \leq l(g^{-1}h) \leq 2 d_T(g,h).$$
\end{corollary}

\begin{proof}
In Section \ref{sec:wordlength} we show that $l(g) = f(g)$ for the function $f$ defined there.  The calculation of the value of $f(g)$ depends only on the coordinates of $\Pi(g)=((m_1(g),l_1(g)), \ldots, (m_d(g),l_d(g)))$. Recall that if $g$ corresponds to the vertex $(g_1, \ldots, g_d)$ in $DL_d(q)$,then for $1 \leq i \leq d$, $$(m_i(g),l_i(g))= (d_{T_i}(o_i,o_i \curlywedge g_i),d_{T_i}(g_i,o_i \curlywedge g_i)),$$ where $(o_1, \ldots, o_d)$ is the vertex in $DL_d(q)$ corresponding to the identity element of $\GG$. Define an analogous relative projection function $\Pi_h(g)= ((m_{h,1}(g), l_{h,1}(g)), \ldots, (m_{h,d}(g), l_{h,d}(g)))$, where for $1 \leq i \leq d$, $$(m_{h,i}(g), l_{h,i}(g))=(d_{T_i}(h_i,h_i \curlywedge g_i),d_{T_i}(g_i,h_i \curlywedge g_i)).$$ Now define $f_h(g)$ as in Section \ref{sec:wordlength}, replacing $\Pi(g)$ with  $\Pi_h(g)$.   Since the proof that $l(g) = f(g)$ is strictly combinatorial, the arguments in Section \ref{sec:wordlength} then imply that $f_h(g)$ computes the word length of $g^{-1}h$ with respect to the generating set $S_{d,q}$, and the Corollary follows directly from Theorem \ref{thm:treedist}.
\end{proof}

The component of the word length function which computes the maximum of the quantities $A_{\sigma(i)}$ over $\sigma \in \Sigma_d$ presents a combinatorial obstruction to writing down a family of geodesic paths representing elements of $\GG$.
The symmetry present in the Diestel-Leader graphs gives rise to a natural family of paths described by edge labels, with the property that any path with these edge labels is a quasi-geodesic path in the Cayley graph $DL_d(q)$.  While it is often not difficult to write down a family of quasi-geodesic paths in a Cayley graph, the paths we describe are very natural paths to traverse and the construction is valid when the trees are permuted, capturing the symmetry of the Diestel-Leader graphs.

Let $g \in \GG$ have projection $\Pi(g) = \mlq$.  Consider the sequence of edge labels
$$(\ed-\eone)^{m_1} (\ed-\etwo)^{m_2} \cdots (\ed-\edd)^{m_{d-1}}(\eone-\ed)^{l_1} (\etwo-\ed)^{l_2} \cdots (\edd-\ed)^{l_{d-1}}(\eone-\ed)^{\alpha}(\ed-\eone)^{l_d}$$
where $\alpha = m_d+(m_1 + \cdots + m_{d-1})-(l_1 + \cdots + l_{d-1}) = m_d + (\sum_{i=1}^d m_i - m_d) -  (\sum_{i=1}^d m_i - l_d) = l_d$.  We claim there is such a path $\zeta_g$ from the basepoint $o$ to the point identified with $g$ in $DL_d(q)$; in general, there are many possible choices of path with the above edge labels.  Moreover, this construction holds under permutation of the trees $T_1,T_2, \cdots T_d$.

%

\begin{corollary}\label{cor:quasigeo}
Let $g \in \GG$ and $\zeta_g$ any path from $\epsilon$ to $\gamma$ with edge labels as listed above.  The $\zeta_g$ is a quasi-geodesic path.
\end{corollary}

\begin{proof}
The corollary follows from combining Theorem \ref{thm:treedist} and Corollary \ref{cor:treedist} and checking that for any two points $h_1$ and $h_2$ along $\zeta_g$ the distance between them along the path $\zeta_g$ is coarsely equivalent to the distance between them in the product metric on the product of trees.
\end{proof}

\section{Dead end elements}

An element in a group $G$ with finite generating set $S$ which corresponds to a vertex $x \in \Gamma(G,S)$ is a {\em dead
end element} if no geodesic ray in $\Gamma(G,S)$ from can be extended past $x$ and remain geodesic. Intuitively, the {\em
depth} of the dead end element $g$ is the length of the shortest path in $\Gamma(G,S)$ from $g$ to any point in the complement of the ball of
radius $l(g)$. Both the existence of dead end elements and their depth are dependent on generating set; in \cite{RW} an
example is given of a finitely generated group which has dead end elements of finite depth with respect to one generating
set, and unbounded depth with respect to another. Theorem \ref{thm:deadend} below generalizes the main result of
\cite{CR1,CR2}, namely that $\Gamma_3(2)$ has dead end elements of arbitrary depth with respect to a generating set similar
to $S_{3,2}$.

\begin{definition}
An element $g$ in a finitely generated group $G$ is a {\em dead end element with respect to a finite generating set $S$} for
$G$ if $l(g)=n$ and $l(gs) \leq n$ for all generators $s$
 in $S \cup S^{-1}$, where $l(g)$ denotes the word length of $g \in G$ with respect to $S$.
\end{definition}

\begin{definition}
A dead end element $g$ in a finitely generated group $G$ with respect to a finite generating set $S$ has {\em depth $k$} if
$k$ is the largest integer with the following property.  If the word length of $g$ is $n$, then $l(g s_1 s_2 \ldots s_r)
\leq n$ for $1 \leq r < k$ and all choices of generators $s_i \in S \cup S^{-1}$.
\end{definition}

The goal of this section is to prove the following theorem.

\begin{theorem}\label{thm:deadend}
The group $\GG$ has dead end elements of arbitrary depth with respect to the generating set $S_{d,q}$.
\end{theorem}

The outline of the proof of Theorem \ref{thm:deadend} mimics the outline of the proof in \cite{CR1,CR2} showing that
$\Gamma_3(2)$ has dead end elements of infinite depth with respect to a generating set similar to $S_{3,2}$.  However, the
details of the proofs are quite different.  In \cite{CR1,CR2} the lamplighter model of an element of $\GG$ is used to compute word length
and analogous lemmas to those below.  This model extends the well-known lamplighter model of an element in $L_n = \Z_n \wr
\Z$ (due to Jim Cannon) in which a group element of $L_n$ is visualized using  a bi-infinite string of multi-state light bulbs placed at integer points along
a number line along with a ``lamplighter."  Then $g \in L_n$ corresponds to a finite collection of illuminated bulbs and an integral position of the lamplighter.  However, in $\Gamma_3(2)$
the ``lampstand" (analogous to $\Z$ for $L_n$) consists of three bi-infinite rays, and the illuminated bulbs are obtained
using a series of relations derived from Pascal's triangle modulo 2, and the ``lamplighter" moves over a $\Z \times \Z$
grid.  A precise extension of this model to describe elements of $\GG$ for $d >3$ seems ambiguous.  The proofs given below
rely instead on the geometry of the Diestel-Leader graphs, and their inherent symmetry.

Begin by defining, for any $n \in \Z^+$, the set $$H_n = \{g \in \GG  ~ | ~\Pi(g) = \mlg $$ $$\text{ with } 0 \leq m_i(g) \leq n
\text{ and } 0 \leq l_i(g) \leq m_i(g)+n \text{ for all } 1 \leq i \leq d \}.$$  In the two lemmas below, we show that the word
length of any point in $H_n$ with respect to $S_{d,q}$ is bounded, and a set of vertices in $H_n$ at maximal distance from
the identity is described.  Proofs of both lemmas follow easily from the wordlength formula proven in Section \ref{sec:wordlength}.

\begin{lemma}\label{lemma:Hn}
If $g \in H_n$ then $l(g) \leq (d+2)n$.
\end{lemma}

\begin{proof}
Let $g \in H_n$ with $\Pi(g) = \mlq$. Choose $\sigma \in \Sigma_d$  so that $l_{\sigma(1)} \geq l_{\sigma(2)} \geq \cdots
\geq l_{\sigma(d)}$. We claim that $m_{\sigma(1)}+ A_{\sigma(i)}(g)+l_{\sigma(d)} \leq (d+2)n$ for every $2 \leq i \leq d$,
hence $f_{\sigma}(g) \leq (d+2)n$.  It then follows from the word length formula that $l(g)  \leq f_{\sigma}(g) \leq (d+2)n$.

Choose $k$ so that $l_{\sigma(k)} > n$, but $l_{\sigma(k+1)} \leq n$, and $k=0$ if $l_{\sigma(i)} \leq n$ for every $i$.
Since $$\sum_{i=1}^d l_{\sigma(i)}(g) = \sum_{i=1}^d m_{\sigma(i)}(g) \leq dn,$$ it follows that $k < d$. Furthermore, we
claim that $l_{\sigma(i)} + \cdots l_{\sigma(d)} \leq (d-i+1)n$ for $1 \leq i \leq d$. This is clear if $i \geq k+1$, since
then each term in the sum is less than $n$. But if $1 \leq i \leq k+1$, then $l_{\sigma(1)}+ \cdots l_{\sigma(i-1)} \geq
(i-1)n$, so $l_{\sigma(i)}+ \cdots l_{\sigma(d)}\leq dn-(-i+1)n=(d-i+1)n$.

For $2 \leq j \leq d-1$,  we see that $m_{\sigma(1)}+ A_{\sigma(j)}(g)+l_{\sigma(d)}=\sum_{i=1}^j m_{\sigma(i)}+
\sum_{i=j}^{d}l_{\sigma(i)}\leq  (j)n+(d-j+1)n=(d+1)n$. But $A_\sigma(d)(g)= \sum_{i=1}^d m _{\sigma(d)} \leq dn$, and hence
$m_{\sigma(1)}+ A_{\sigma(d)}(g)+l_{\sigma(d)} \leq (d+2)n$. Thus, $m_{\sigma(1)}+ A_{\sigma(i)}(g)+l_{\sigma(d)} \leq (d+2)n$ for
every $2 \leq i \leq d$, as claimed, and the lemma follows.
\end{proof}

The next lemma follows immediately from the word length formula of Section \ref{sec:wordlength}.
\begin{lemma}\label{lemma:5n}
If $g_n \in H_n$ and $\Pi(g_n) = \left( (n,n)(n,n) \cdots (n,n) \right)$ then $l(g) = (d+2)n$.
\end{lemma}

The proof of Theorem \ref{thm:deadend} follows easily from Lemmas \ref{lemma:Hn} and \ref{lemma:5n}.

{\it Proof of Theorem \ref{thm:deadend}.} Let $g_n \in H_n$ be any element with $\Pi(g_n) = ((n,n)(n,n) \cdots (n,n))$.  In
Lemma \ref{lemma:5n} it is shown that $l(g)=(d+2)n$.  It follows immediately from Lemma \ref{lemma:Hn} that $g_n$ is a dead
end element, as all vertices adjacent to $g_n$ lie in $H_n$.

To see that the depth of $g_n$ is at least $n$, note that the length of a path from $g_n$ to a point outside $H_n$ must
contain a subpath of at least $n$ edges.  Thus the depth of $g_n$ is at least $n$ and we conclude that $\GG$ has dead end
elements of arbitrary depth with respect to the generating set $S_{d,q}$. \qed

\section{Cone types and geodesic languages}
We now prove that $\GG$ has no regular language of geodesics with respect to the generating set $S_{d,q}$, that is, there is
no collection of geodesic representatives for elements of $\GG$ which is accepted by a finite state automata.  The existence
of a regular language of geodesics for a finitely generated group $G$ is equivalent to the finiteness of the set of cone
types of $G$.  It is a well known theorem in computer science that a language is regular if and only if it has finitely many
distinct left quotients.  In the case of a geodesic language, the left quotients are exactly the cone types.  We prove that
$\GG$ has infinitely many cone types with respect the generating set $S_{d,q}$, and it follows that $\GG$ has no regular
language of geodesics with respect to $S_{d,q}$.

We begin by defining the {\em cone} and the {\em cone type} of an element $g \in G$, where $G$ is a group with finite
generating set $S$.  Cannon defined the cone type of an element $w \in G$ to be the set of geodesic
extensions of $w$ in the Cayley graph $\Gamma(G,S)$. \cite{cannoncone}

\begin{definition}
A path $p$ is {\em outbound} if $d(1,p(t))$ is a strictly increasing function of $t$. For a given $g\in G$, the {\em cone}
at $g$, denoted $C'(g)$ is the set of all outbound paths starting at $g$. Define the {\em cone type} of $g$, denoted $C(g)$,
to be $g^{-1}C'(g)$.
\end{definition}

This definition applies both in the discrete setting of the group and in the one-dimensional metric space which is the
Cayley graph. A subtlety is that if the presentation for $G$ includes odd length relators, then the cone type of an element
in the Cayley graph may include paths which end at the middle of an edge.  If the presentation for $G$ consists entirely of
even length relators, then every cone type viewed in the Cayley graph consists entirely of full edge paths.  We refer the
reader to \cite{NS} for a more detailed discussion of cone types.

\begin{theorem}\label{thm:conetypes}
The group $\GG$ has infinitely many cone types with respect to the generating set $S_{d,q}$.
\end{theorem}

The following corollary is an immediate consequence of Theorem \ref{thm:conetypes}.

\begin{corollary}\label{cor:languages}
The group $\GG$ has no regular language of geodesics with respect to the generating set $S_{d,q}$.
\end{corollary}

We begin with a lemma stating sufficient but not necessary conditions on $\sigma \in \Sigma_d$ which ensure that $f(g) = f_{\sigma}(g)$; this lemma will
be extremely useful in the the proof of Theorem \ref{thm:conetypes} as realizing when $f(g)=f_{\sigma}(g)$ for a particular
$g \in \GG$ and $\sigma \in \Sigma_d$ can be quite difficult.  Recall that we identify $g \in \GG$ with the vertex $x \in
DL_d(q)$ corresponding to it, and abuse notation by writing $\Pi(g)$ for $\Pi(x)$.

\begin{lemma}\label{lemma:max-m}
Let $g \in \GG$ have projection $\Pi(g) = \mlg$.  If $\sigma \in \Sigma_d$ satisfies
\begin{enumerate}
\item $\min_{\tau \in \Sigma_d} m_{\tau(1)}(g) + l_{\tau(d)}(g) = m_{\sigma(1)}(g) + l_{\sigma(d)}(g) $, and \item
    $\max_{2 \leq i \leq d} A_{\sigma(i)}(g) = A_{\sigma(d)}(g)$
\end{enumerate}
then $f(g) = f_{\sigma}(g)$.
\end{lemma}

\begin{proof}
Let $\sigma$ be as in the statement of the lemma, and $\tau$ any element of $\Sigma_d$.  It is always true that
 $\max_{2 \leq i \leq d} A_{\tau(i)}(g) \geq A_{\tau(d)}(g)$, and that $m_{\tau(1)}(g) + l_{\tau(d)}(g) \geq
 m_{\sigma(1)}(g) + l_{\sigma(d)}(g)$ by the choice of $\sigma$.

Hence
\begin{align*}f_{\tau}(g) &=m_{\tau(1)}(g) + l_{\tau(d)}(g) +  \max_{2 \leq i \leq d} A_{\tau(i)}(g)\\
  & \geq m_{\sigma(1)}(g) + l_{\sigma(d)}(g) + A_{\tau(d)}(g) \\
  & =m_{\sigma(1)}(g) + l_{\sigma(d)}(g) + A_{\sigma(d)}(g)\\
  & = f_{\sigma}(g).
  \end{align*}

Thus by the definition of $f(g)$ we must have $f(g) = f_{\sigma}(g)$.

\end{proof}

To prove Theorem \ref{thm:conetypes} we define a sequence of elements \{$g_n\}$ so that there is a geodesic path of length
$n$ from $g_n$ terminating at a dead end element, and so that no shorter geodesic path from $g_n$ reaches any other dead end
element of the group.  Thus each $g_n$ lies in a different cone type, and the theorem follows.

{\it Proof of Theorem \ref{thm:conetypes}.} Let $g_n$ for $n \in \Z^+$ be any element with projection
$$\Pi(g_n) = \left( (2n,3n),(3n,4n),(4n,5n),(5n,6n), \cdots ,((d-1)n,dn),(dn,3n),(2n,n)\right).$$
We first show that $f(g_n) = f_{\epsilon}(g_n)$ where $\epsilon$ is the identity permutation, and specifically that $f(g_n)
= 3n + \sum_{i=1}^d m_i(g_n)$.

First note that $\min_{\tau \in \Sigma_d} m_{\tau(1)}(g_n) + l_{\tau(d)}(g_n) = 3n = m_{\epsilon(1)}(g_n) +
l_{\epsilon(d)}(g_n) $.  Second, consider $A_{\epsilon(d)}(g_n) = 2n+ \sum_{j=2}^{d} jn=4n+ \sum_{j=3}^{d} jn$ and compare
this value to $A_{\epsilon(i)}(g_n)$ for $i \neq d$.  When $2 \leq i < d-1$,
\begin{align*}
A_{\epsilon(i)}(g_n) &= [m_2(g_n) + m_3(g_n) + \cdots + m_i(g_n)] + [l_i(g_n) + \cdots + l_{d-1}(g_n)]  \\
 &= [3n + 4n + \cdots (i+1)n] + [(i+2)n + \cdots + dn + 3n] \\
 &= 3n+ \sum_{j=3}^{d} jn < 4n + \sum_{j=3}^{d} jn = A_{\epsilon(d)}(g_n).
\end{align*}
When $i=d-1$ we see that $A_{d-1}(g_n) =  3n+ \sum_{j=3}^{d} jn < 4n + \sum_{j=3}^{d} jn = A_{\epsilon(d)}(g_n).$ It then
follows from Lemma \ref{lemma:max-m} that $f(g_n) = f_{\epsilon}(g_n) = 3n + \sum_{i=1}^d m_i(g_n)$.  We note for later use
that $A_{\epsilon(d)}(g_n) - A_{\epsilon(i)}(g_n) = n$ when $i \neq d$.

Let $h_n$ be any point connected to $g_n$ by a path of length at most $n$ in $DL_d(q)$.  Then $h_n$ has projection
\begin{align*} \Pi(h_n) &= ( (2n,3n-r_1),(3n,4n-r_2),(4n,5n-r_3),(5n,6n-r_4), \cdots  \\ &
((d-1)n,dn-r_{d-2}),(dn,3n-r_{d-1}),(2n,n-r_d) )
\end{align*}
where the $r_i$ satisfy:
\begin{enumerate}
\item  $\sum_{i=1}^d r_i = 0$, and \item the sum of the positive $r_i$ is at most $n$; hence the sum of the negative
    $r_i$ is at least $-n$.
\end{enumerate}

We first calculate $f(h_n)$, again using Lemma \ref{lemma:max-m}. Note that for any $\tau \in \Sigma_d$,
$$\min_{\tau \in \Sigma_d} m_{\tau(1)}(h_n) + l_{\tau(d)}(h_n) = 3n-r_d = m_{\epsilon(1)}(h_n) + l_{\epsilon(d)}(h_n) $$ and
that $2n \leq 3n-r_d \leq 4n$.  Moreover, $A_{\epsilon(d)}(h_n) = A_{\epsilon(d)}(g_n)$.  We compare $A_{\epsilon(i)}(g_n)$
and $A_{\epsilon(i)}(h_n)$ for $i \neq d$, and see that
\begin{align*}
A_{\epsilon(i)}(h_n) &= 3n + 4n + \cdots + (i+1)n + (i+2)n-r_{i} + (i+3)n-r_{i+1} + \cdots dn - r_{d-2} + 3n-r_{d-1} \\
 &=A_{\epsilon(i)}(g_n) - (r_{i} + \cdots + r_{d-1}) \leq A_{\epsilon(i)}(g_n) + n.
\end{align*}
Above we saw that $A_{\epsilon(d)}(g_n) - A_{\epsilon(i)}(g_n) =n$ for $2 \leq i<d$.  Combining this with the above
inequality yields
$$A_{\epsilon(i)}(h_n) \leq A_{\epsilon(i)}(g_n) + n = A_{\epsilon(d)}(g_n) -n +n = A_{\epsilon(d)}(g_n)  =
A_{\epsilon(d)}(h_n)$$
and hence $\max_{2 \leq i \leq d} A_{\epsilon(i)}(h_n) = A_{\epsilon(d)}(h_n)$.  Lemma \ref{lemma:max-m} then implies that
$$f(h_n) = f_{\epsilon}(h_n) = 3n-r_d + A_{\epsilon(d)}(h_n).$$

Now choose $h_n$ to be a point of the above form which is connected to $g_n$ by a path of length at most $n-1$.  We show
that $h_n$ is not a dead end element by exhibiting a generator $s$ so that $f(h_ns) = f(h_n) + 1$.  Let $s \in S_{d,q}$ be a
generator corresponding to an edge of type $\ed-\eone$ emanating from $h_n$, so that
\begin{align*} \Pi(h_ns) &= ( (2n,3n-r_1-1),(3n,4n-r_2),(4n,5n-r_3),(5n,6n-r_4), \cdots  \\ & \cdots
,((d-1)n,dn-r_{d-2}),(dn,3n-r_{d-1}),(2n,n-r_d+1) )
\end{align*}
As the ordered pairs in the projection are unchanged between $\Pi(h_n)$ and $\Pi(h_ns)$ except in the second coordinate of
the first and last ordered pairs, it is still the case that $\max_{2 \leq i \leq d} A_{\epsilon(i)}(h_ns) =
A_{\epsilon(d)}(h_ns)$. Note as well that
$$\min_{\tau \in \Sigma_d} m_{\tau(1)}(h_ns) + l_{\tau(d)}(h_ns) = 3n-r_d+1 = m_{\epsilon(1)}(h_ns) + l_{\epsilon(d)}(h_ns).
$$
The maximum value of $3n-r_d+1$ is $4n$; it may be possible to achieve a value of $4n$ using another permutation
in $\Sigma_d$, but if $3n-r_d+1 =4n$, the value of $m_{\tau(1)}(h_ns) + l_{\tau(d)}(h_ns)$ can never be less than $4n$ with
any non-identity permutation.  Thus we can achieve the minimum value of this quantity using $\epsilon$.  Lemma
\ref{lemma:max-m} now implies that $f(h_ns) = f_{\epsilon}(h_ns) = 3n-r_d+1 + A_{\epsilon(d)}(h_ns) = 3n-r_d+1
+A_{\epsilon(d)}(h_n) = f(h_n) + 1$ and thus $h_n$ is not a dead end element in $\GG$ with respect to the generating set
$S_{d,q}$.

We now show that there is a geodesic path of length $n$ from $g_n$ which terminates at a dead end element which we denote
$g_{n,n}$.  Namely, consider any path of length $n$ originating at $g_n$ with the property that the $i$-th point on the
path, denoted $g_{n,i}$, has projection
$$\Pi(g_{n,i}) =  ( (2n,3n-i),(3n,4n),(4n,5n),(5n,6n), \cdots,((d-1)n,dn),(dn,3n),(2n,n+i))$$
for $1 \leq i \leq n$.  Letting $r_1 = i$, $r_d = -i$ and $r_j = 0$ for $1 < j < d$, the above argument implies that
$$f(g_{n,i}) = f_{\epsilon}(g_{n,i}) = 3n-r_d +A_{\epsilon(d)}(g_{n,i}) = 3n-r_d + A_{\epsilon(d)}(g_{n})=f(g_n) - r_d =
f(g_n) + i$$ and hence this path is geodesic.

We now show that the endpoint $g_{n,n}$ of this path, which has projection
$$\Pi(g_{n,n}) =  \left( (2n,2n),(3n,4n),(4n,5n),(5n,6n), \cdots ,((d-1)n,dn),(dn,3n),(2n,2n)\right)$$
is a dead end element in $\GG$ with respect to the generating set $S_{d,q}$.

We know that $f(g_{n,n}) =  4n + A_{\epsilon(d)}(g_{n,n})$.  Let $s \in S_{d,q}$ be any generator so that $g_{n,n}s \neq
g_{n,n-1}$.  We must show that $f(g_{n,n}s) \leq f(g_{n,n})$.  Since $l_{i}(g_{n,n}) > 0$ for all $i$, there must be indices
$j \neq k$ so that

\begin{enumerate}
\item $(m_j(g_{n,n}s),l_j(g_{n,n}s)) = (m_j(g_{n,n}),l_j(g_{n,n})+1)$, \item $(m_k(g_{n,n}s),l_k(g_{n,n}s)) =
    (m_k(g_{n,n}),l_k(g_{n,n})-1)$, and \item $(m_r(g_{n,n}s),l_r(g_{n,n}s)) = (m_r(g_{n,n}),l_r(g_{n,n}))$ for $r \neq
    j,k$.
\end{enumerate}

Case 1: $k=d$. Using the identity permutation $\epsilon$, note that
$$\min_{\tau \in \Sigma_d} m_{\tau(1)}(g_{n,n}s) + l_{\tau(d)}(g_{n,n}s) = 4n-1 = m_{\epsilon(1)}(g_{n,n}s) +
l_{\epsilon(d)}(g_{n,n}s).$$
It may now be the case that $A_{\epsilon(i)}(g_{n,n}s) = A_{\epsilon(i)}(g_{n,n})+1$ for some $i$; however, it is always
true that for $2 \leq i \leq d-1$ we have $A_{\epsilon(i)}(g_{n,n}s)  \leq A_{\epsilon(i)}(g_{n,n})+1$.   Since
$A_{\epsilon(d)}(g_{n,n})-A_{\epsilon(i)}(g_{n,n})=n$, we see that for $2 \leq i \leq d-1$
$$A_{\epsilon(i)}(g_{n,n}s) \leq A_{\epsilon(i)}(g_{n,n})+1 \leq A_{\epsilon(i)}(g_{n,n})+n = A_{\epsilon(d)}(g_{n,n}) =
A_{\epsilon(d)}(g_{n,n}s).$$
It then follows from Lemma \ref{lemma:max-m} that $f(g_{n,n}s) = f_{\epsilon}(g_{n,n}s) = 4n-1 +A_{\epsilon(d)}(g_{n,n}s)$.
Since  $A_{\epsilon(d)}(g_{n,n}s) =A_{\epsilon(d)}(g_{n,n})$ we see that $f(g_{n,n}s) =f(g_{n,n})-1$.

Case 2: $k=1$. Replacing $\epsilon$ with the permutation $\sigma = (1 \ d) \in \Sigma_d$, the argument in Case 1 shows that
$f(g_{n,n}s) =f_{\sigma}(g_{n,n}s)=f(g_{n,n})-1$.

Case 3: $2 \leq k \leq d-1$ and $j \neq d$. First note that
$$\min_{\tau \in \Sigma_d} m_{\tau(1)}(g_{n,n}s) + l_{\tau(d)}(g_{n,n}s) = 4n = m_{\epsilon(1)}(g_{n,n}s) +
l_{\epsilon(d)}(g_{n,n}s)$$
and that $A_{\epsilon(d)}(g_{n,n}s) =A_{\epsilon(d)}(g_{n,n}).$  As in the above cases, for $2 \leq i \leq d-1$ we have
$A_{\sigma(i)}(g_{n,n}s)  \leq A_{\epsilon(i)}(g_{n,n})+1$ and the same reasoning yields $A_{\epsilon(i)}(g_{n,n}s) \leq
A_{\epsilon(d)}(g_{n,n}s).$  Together this shows that $f(g_{n,n}s) = f_{\epsilon}(g_{n,n}s) = 4n
+A_{\epsilon(d)}(g_{n,n}s)$.  Since  $A_{\epsilon(d)}(g_{n,n}s) =A_{\epsilon(d)}(g_{n,n})$ we see that $f(g_{n,n}s)
=f(g_{n,n})$.

Case 4: $2 \leq k \leq d-1$ and $j = d$.  Replacing $\epsilon$ with the permutation $\sigma = (1 \ d) \in \Sigma_d$, the
argument in Case 3 shows that $f(g_{n,n}s)=f_{\sigma}(g_{n,n}s) =f(g_{n,n})$.

Combining the above four cases shows that $f(g_{n,n}s) \leq f(g_{n,n})$ for all $s \in S_{d,q}$ and hence $g_{n,n}$ is a
dead end element in $\GG$ with respect to this generating set.

Thus, there is a geodesic path of length $n$ from $g_n$ which terminates at a dead end element of $\GG$, and no shorter path
from $g_n$ reaches a dead end element.  Hence each $g_n$ lies in a distinct cone type, and the theorem follows. \qed

\bibliographystyle{plain}
\bibliography{refs}

\end{document}